\documentclass{amsart}
\usepackage{amsmath,amssymb,amsthm,microtype,graphicx,array}
\usepackage[hidelinks]{hyperref}
\usepackage[nameinlink]{cleveref}

\let\cref\Cref

\usepackage{enumitem}

\newcommand{\Z}{\mathbb{Z}}

\newcommand{\R}{\mathbb{R}}
\newcommand{\CC}{\mathcal{C}}
\newcommand{\pr}{^{\prime}}

\usepackage{xcolor}

\title{The Conway knot has infinite concordance order}
\author{Chiara Donatone}
\address{School of Mathematics and Statistics, University of Glasgow, University Place, Glasgow, G12 8QQ, United Kingdom}
\email{donatonechiara@gmail.com}
\author{Marc Kegel}
\address{Universidad de Sevilla, Dpto.\ de \'Algebra,
41012 Sevilla, Spain}
\email{kegelmarc87@gmail.com}
\urladdr{\url{https://marckegel.github.io}}
\author{Lukas Lewark}
\address{Department of Mathematics, ETH Zurich, 8092 Zurich, Switzerland}
\email{lukas.lewark@math.ethz.ch}
\urladdr{\url{https://people.math.ethz.ch/~llewark/}}
\author{Paula Truöl}
\address{School of Mathematics and Statistics, University of Glasgow, University Place, Glasgow, G12 8QQ, United Kingdom}
\email{paula.truol@glasgow.ac.uk,
paulagtruoel@gmail.com}
\urladdr{\url{https://paulatruoel.github.io}}

\newtheorem{theorem}{Theorem}
\newtheorem{lemma}[theorem]{Lemma}

\newtheorem{prop}[theorem]{Proposition}
\newtheorem{definition}[theorem]{Definition}
\newtheorem{question}[theorem]{Question}

\makeatletter
\newtheorem*{rep@theorem}{\rep@title}
\newcommand{\newreptheorem}[2]{%
\newenvironment{rep#1}[1]{%
 \def\rep@title{#2 \ref{##1}}%
 \begin{rep@theorem}}%
 {\end{rep@theorem}}}
\makeatother

\newreptheorem{theorem}{Theorem}
\newreptheorem{prop}{Proposition}

\theoremstyle{remark}
\newtheorem{remark}[theorem]{Remark}

\DeclareMathOperator{\lk}{lk}

\newcommand{\qua}{\hskip 0.4em \ignorespaces}
\def\arxiv#1{\relax\ifhmode\unskip\qua\fi
\href{http://arxiv.org/abs/#1}%
{\tt arXiv:\penalty -100\unskip#1}}

\def\MR#1{\relax\ifhmode\unskip\qua\fi
\href{https://mathscinet.ams.org/mathscinet-getitem?mr=#1}{\tt MR#1}}
\def\ZB#1{\relax\ifhmode\unskip\qua\fi
\href{https://zbmath.org/?q=an:#1}{\tt Zbl\:#1}}
\def\xox#1{\csname xx#1\endcsname}

\renewenvironment{thebibliography}[1]{
  \begin{oldthebibliography}{ABC12}\small
    \setlength{\itemsep}{.5ex}
    \setlength{\parskip}{0em}
}
{
  \end{oldthebibliography}
}

\begin{document}
\begin{abstract}
We examine how the Rasmussen invariant, satellite operations, and null-homologous twists can be used to establish infinite order of knots in the smooth concordance group. As an application, we show that the Conway knot has infinite concordance order.
\end{abstract}
\maketitle

\section{Introduction}

A knot in the $3$-sphere $S^3$ is \emph{slice} if it bounds a smoothly and properly embedded disk in~$B^4$, the $4$-ball bounded by~$S^3$. Two knots $K$ 
and $J$ are \emph{concordant} if the connected sum $K \# {-J}$ is slice. Here $-J$ is the \emph{inverse} of~$J$, which is the 
mirror image of $J$ with the opposite orientation.
Knots up to concordance form an abelian group, the \emph{concordance group}~$\CC$. The group operation is induced by connected sum, slice knots represent the identity element, and the additive inverse of a concordance class $[K]$ is~$[-K]$. While the study of the group~$\CC$ has advanced in recent years, particularly since the advent of Heegaard Floer and Khovanov homology theories, its structure---let alone its isomorphism type---remains mysterious. In particular, the structure of the torsion subgroup of~$\CC$ is an open problem (see e.g.~\cite[Problem~1.38]{K3_list}). The \emph{concordance order} of a knot $K$ is its order in~$\CC$, i.e.~the least positive integer $n$ for which the connected sum $K^{\#n}$ of~$n$ copies of~$K$ is slice, or infinity if no such $n$ exists.

\begin{theorem}\label{thm:Conway_inf_order}
The Conway knot has infinite concordance order.
\end{theorem}

It had previously been known that the Conway knot $C$ is not slice, as demonstrated in the celebrated work of Piccirillo~\cite{piccirillo:Conway}. Piccirillo's proof uses a clever construction involving finding a knot $K\pr$ with the same $0$-trace as~$C$. According to the trace embedding lemma (see \Cref{sec:twists}), this gives $K\pr$ the same sliceness status as~$C$. However, computing Rasmussen's $s$ invariant~\cite{rasmussen} for $K\pr$ shows that $K\pr$ is not slice. 
Piccirillo's argument settles the sliceness of $C$ itself, but says nothing about the connected sums $C^{\#n}$, and so leaves the concordance order of $C$ open.

One difficulty in determining the non-sliceness of the Conway knot was that 
all known concordance invariants vanish on $C$; in particular, the additive ones,
which could otherwise have been used to prove \Cref{thm:Conway_inf_order}.  
For this reason, we use an additional technique called \emph{twisting}.

\begin{definition}\label{def:twist}
For knots $J, K\subset S^3$, we say that $K$ is obtained from $J$ by a \emph{right-handed null-homologous twist},
or simply, a \emph{twist}, if there is an unknot $\alpha \subset S^3\setminus J$ with $\lk(\alpha, J) = 0$ such that $(-1)$-surgery on $S^3$ along $\alpha$ turns $J$ into~$K$. 
We say that a real-valued knot invariant $y$ satisfies the \emph{twist inequality} if $y(K) \leq y(J)$ whenever~$K$ is obtained from $J$ by a twist.
\end{definition}

For example, a positive-to-negative crossing change is a special type of twist. (This is why twisting is also called a \emph{positive-to-negative generalised crossing change}, or \emph{adding a generalised negative crossing} \cite{zbMATH06339370}.) 
Alongside many other knot concordance invariants, the Rasmussen $s$ invariant satisfies the twist inequality (see \Cref{sec:twists}). Another key element in our proof of \Cref{thm:Conway_inf_order} is the behaviour of satellite operations under twists. Recall that a \emph{pattern}~$P$ is a knot in the standard solid torus~$S^1 \times D^2$, and $P(K)$ denotes the satellite knot with pattern~$P$ and companion~$K$; see, for example, \cite[Section 4D]{zbMATH00877622}. 

\begin{theorem}\label{thm:twist_ineq}
Let $J, K\subset S^3$ be two knots that can be transformed into the unknot~$U$ via a finite sequence of concordances and twists. 
If a knot concordance invariant $y$ exists that satisfies the twist inequality and a pattern $P$ such that ${y(P(K)) \neq y(P(U))}$, then $K \# J$ is not slice. 
In particular, $K$ has infinite concordance order.
\end{theorem}

To prove \Cref{thm:Conway_inf_order} using \Cref{thm:twist_ineq}, we must first find a suitable pattern~$P$. Building on Piccirillo's work~\cite{piccirillo:shake-genus,piccirillo:Conway}, 
for every knot~$K$ with unknotting number~$1$ and an unknotting crossing~$c$, we construct a pattern~$P_{K,c}$ 
such that the satellite knot~$P_{K,c}(J)$ has the same $0$-trace as~$K \#J$ for every knot~$J$. By the trace embedding lemma, the sliceness statuses of $K\#J$ and $P_{K,c}(J)$ are identical. 

Experiments suggest that $P_{K,c}(U) \neq K$ unless $K$ is a twist knot or a twisted Whitehead double (see, for example, \cite[Section~5]{KS_exotic_traces}). We are particularly interested in the pattern~$P_{K,c}$ when it additionally satisfies~$s(P_{K,c}(U))\neq 0$.

\begin{prop}\label{cor:inf_order}
    Let $K$ be a knot with unknotting number~$1$ and an 
    unknotting crossing~$c$. If $s(P_{K,c}(U))\neq 0$, then $K$ has infinite concordance order.
\end{prop}

According to KnotInfo~\cite{knotinfo}, the only prime knot with at most~$11$ crossings for which the concordance order is currently unknown is the Conway knot~$C$.
Note that we follow the same convention as in \cite{piccirillo:Conway}, where $C$ is the mirror of the knot~$K11n34$ from~\cite{Table}, cf.~\cite{DTnotation,knotinfo}.
\Cref{thm:Conway_inf_order} thus concludes the computation of concordance orders for knots with at most 11 crossings.
However, the following question remains open.

\begin{question}\label{question:Z-summand}
Does the concordance class $[C]$ of the Conway knot $C$
generate a $\mathbb{Z}$-summand of~$\mathcal{C}$?
\end{question}

The infinite concordance order of a knot is often demonstrated using knot concordance homomorphisms that do not vanish for the knot in question. In fact, the concordance class $[K]$ of a given knot $K$ generates a $\mathbb{Z}$-summand of~$\mathcal{C}$ if and only if there is a knot concordance homomorphism to $\mathbb{Z}$ that takes the value~$\pm 1$ on~$K$. We do not know if the Conway knot admits such a homomorphism.

We note that, subsequent to the appearance of our preprint, Golla--Pinzón-Caicedo independently proved \Cref{thm:Conway_inf_order} using different methods~\cite{golla_pinzon-caicedo}. They moreover study an infinite family of \emph{generalised Conway knots} $C_{r,n}$ (see for example the discussion in~\cite{Baldwin2025ConwayKnots}) and prove, in particular, that $C_{2,n}$ has infinite concordance order for every $n\neq 0$~\cite[Theorem~1.3]{golla_pinzon-caicedo}. In \Cref{rem:generalised_Conway_knots} below, we explain that our methods also apply to this family and give an alternative proof of their result. 

\begin{remark}\label{rmk:independent}
  We can use \Cref{thm:Conway_inf_order} to obtain linear independence statements. Indeed, suppose that $K$ is a knot admitting a concordance homomorphism $y\colon\mathcal{C}\to\mathbb{R}$ with~$y(K)\neq 0=y(C)$. Given a trivial linear combination $a[K]+b[C]$ in concordance for integers $a,b$, applying $y$ gives $a\,y(K)=0$ and hence $a=0$. So $b[C]$ must be trivial in concordance, which implies $b=0$ by \Cref{thm:Conway_inf_order}. Thus $[K]$ and~$[C]$ are linearly independent. Taking $y=s$, this applies e.g.~to every non-trivial strongly quasipositive knot~$K$, since $s(K)=2g(K)>0$~\cite{zbMATH05274876} and $s(C)=0$.

  Similarly, all Levine--Tristram signatures of $C$ vanish, since $C$ has a trivial Alexander polynomial. Combining this with the linear independence of the torus knots established by Litherland~\cite{litherland} and with   \Cref{thm:Conway_inf_order}, we obtain that $\{[C]\}\cup \{[T_{p,q}] \mid 2 \le p < q,\ \gcd(p,q)=1\}$ 
  is a linearly independent subset of~$\mathcal{C}$.
\end{remark}

\subsection*{Organisation}
We prove \Cref{thm:twist_ineq} in \Cref{sec:twists}. \Cref{sec:picc_patterns} is devoted to constructing and defining suitable patterns~$P_{K,c}$, and proving \Cref{cor:inf_order} and \Cref{thm:Conway_inf_order}. 
\Cref{sec:experiments} describes our experiments with knots of low crossing number.

\subsection*{Acknowledgements}
We thank Peter Feller, Lisa Piccirillo, and Claudius Zibrowius for helpful remarks, Marco Golla for comments on a first draft, 
Livio Ferretti for useful discussions on related topics and Evan Scott for a question that led to \Cref{rmk:independent}.
Part of this work was carried out when CD and MK visited~LL at ETH Zürich. We thank ETH for the hospitality. CD and PT would like to thank the Max Planck Institute for Mathematics in Bonn for the hospitality during the initial stages of this project.
CD acknowledges support from the European Union's Horizon Europe research and innovation programme under the Marie Skłodowska-Curie grant agreement No 101223504 (TopSCoLDS Doctoral Network).
MK is supported by a Ram\'on y Cajal grant \mbox{(RYC2023-043251-I)} and PID2024-157173NB-I00 funded by MCIN/AEI/10.13039/5\-01100011033, by ESF+, and by FEDER, EU; and by a VII Plan Propio de Investigación y Transferencia (SOL2025-36103) of the University of Seville. 
PT acknowledges the support of the Swiss National Science Foundation Postdoc.Mobility fellowship 230329, and thanks the University of Glasgow for its hospitality and support.

\section{Twists, monotonicity, and non-sliceness}\label{sec:twists}

Unless stated otherwise, all manifolds are assumed to be oriented, compact, and smooth throughout. We consider manifolds up to orientation-preserving diffeomorphism and links up to isotopy. Let us begin with the definition of a relation on the concordance group~$\mathcal{C}$.
Similar relations have been considered e.g.~in~\cite{zbMATH04096384,zbMATH06198082,zbMATH06339370}.

\begin{definition}
For concordance classes of knots $J, K\subset S^3$, we write $[K] \preceq [J]$ if $K$ is obtained from $J$ by a finite sequence of concordances and twists (see \Cref{def:twist}).
\end{definition}

\begin{remark}
Equivalently, $[K] \preceq [J]$ if and only if there is a smooth, proper, null-homologous embedding $\Sigma\colon S^1\times [0,1] \to (S^3 \times [0,1]) \# (\mathbb{C}P^2)^{\#m}$ for some $m \geq 0$ such that $\Sigma(S^1\times\{0\}) = -K \subset S^3 \times \{0\}$ 
and $\Sigma(S^1\times\{1\}) = J \subset S^3 \times \{1\}$. Similar statements are well known, see e.g.~\cite[Theorem~5.7]{zbMATH06339370}. Since we are not going to rely on this characterisation, we only provide a proof sketch. For the `only if' direction, note that if $K$ is obtained from $J$ by a twist, then blowing up once yields a cobordism as above with~$m = 1$; composing concordances and such cobordisms produces the desired~$\Sigma$. For the `if' direction, decompose a given~$\Sigma$ as a sequence of cups, saddles, caps, and blow-ups of the ambient manifold. Up to isotopy, we can assume that cups come first and caps last, and we can permute saddles and blow-ups arbitrarily. In this way, we may obtain~$K$ from~$J$ by a sequence consisting of a concordance, $m$ twists, and another concordance. (In fact, the first concordance is ribbon, and the second concordance is ribbon when seen upside down.)
\end{remark}

The relation $\preceq$ is reflexive and transitive, but not antisymmetric; for instance, $[4_1] \preceq [U]$ and $[U] \preceq [4_1]$ hold for the figure-eight knot~$4_1$. Moreover, $\preceq$ is translation-invariant, i.e.~for all knots $J$, $K$  and~$M$,
\begin{align*}
    [K] \preceq [J] \quad \Leftrightarrow \quad
[M\#K] \preceq [M\#J].
\end{align*}

Next, we consider functions $y\colon \CC\to\mathbb{R}$ (not necessarily homomorphisms) that are \emph{monotone}, i.e.~that satisfy 
\begin{equation}\label{eq:ypreserves}
[K] \preceq [J] \quad \Rightarrow \quad y(K) \leq y(J).
\end{equation}

Our primary interest lies in the Rasmussen invariant $s\colon \CC\to\mathbb{Z}$~\cite{rasmussen}, which is monotone. More precisely, there is a Rasmussen invariant $s_p$ for each possible characteristic $p$ of the underlying field (with the original Rasmussen invariant~$s$ equal to~$s_0$). All of them are monotone, since they satisfy an adjunction inequality in the punctured connected sum of~$\overline{\mathbb{C}P}^2$s, as shown in \cite[Corollary~1.9 and Theorem~1.11]{manolescu-marengon-sarkar-willis} and \cite[Corollaries 1.4 and~1.5]{ren:lee_torus_links}.

\begin{remark}
Other monotone invariants include:
\begin{itemize}
\item the spatially refined versions of~$s$~\cite{ren2025adjunctioninequalityspatiallyrefined},
\item the invariants $\vartheta_{s_p}$, which control $s_p$ of twisted Whitehead doubles \cite[Prop.~2.34]{zbMATH07939044},
\item the invariant $s^{\#}$ \cite{zbMATH06272212},
\item the invariant~$\tau$ \cite[Theorem~1.1]{zbMATH02057402},
\item the invariant $\nu$ from \cite{os:rational_surgeries}, see \cite[Theorem~4.7]{zbMATH07408046},
\item the invariants $\nu^+$, $-\Upsilon(t)$, $-d(S^3_{p/q}(\cdot),i)$ and $V_k$ from \cite{hom_wu,os:conc_homos,os:d_inv,rasmussen:hi,ni-wu} for~$t\in[0,2]$, coprime integers $p,q > 0$ and $0 \leq i \leq p-1$, and $k \in \Z_{\geq 0}$, as 
follows from Theorem~1.2 in \cite{SATO2018113} for~$\nu^+$, and from combining Proposition~1.5~(3) and Theorem~1.6~(1) in the same paper for the other mentioned invariants,
\item and $-\sigma_\omega$, for $\sigma_\omega$ the classical Levine--Tristram signatures from~\cite{Levine1969,tristram_1969} for $\omega \in S^1$ with non-vanishing Alexander polynomial $\Delta(\omega) \neq 0$; see, for example, \cite[Theorem~3.16]{zbMATH04096384}.
\end{itemize}
Monotonicity appears to be unknown for $\mathfrak{sl}_n$-versions of the Rasmussen invariant for~$n\geq 3$ as defined in~\cite{zbMATH05614872,zbMATH05550658}.
Monotonicity does not hold for the Heegaard Floer $\varepsilon$-invariant \cite{zbMATH06366498}. For example, for the satellite knot $K$ with the Mazur pattern and companion the negative trefoil knot, we have $[K] \preceq [U]$ by \cref{lem:monotonesattelite} below. However, $\varepsilon(K) = 1$ by~\cite{zbMATH06700148}.
\end{remark}

\begin{remark}
For more general definite $4$-manifolds, an adjunction inequality, as it is known for $\tau$ and $s^{\#}$~\cite{zbMATH02057402,zbMATH06272212}, has not been established for~$s$; see~\cite[Question~9.7]{manolescu-marengon-sarkar-willis}. This is the reason why we consider the relation~$\preceq$, which differs from similar relations defined previously via the existence of cobordisms in certain definite $4$-%
manifolds~\cite{zbMATH04096384,zbMATH06198082}.
\end{remark}

Every pattern $P$ in the solid torus~$S^1 \times D^2$ gives rise to a function $P\colon\CC \to\CC$, $P([K]) = [P(K)]$, called a \emph{satellite operation}. Satellite operations are also monotone; namely, we have
\begin{equation}\label{eq:Ppreserves}
[K] \preceq [J] \quad \Rightarrow \quad
[P(K)] \preceq [P(J)].
\end{equation}

This is a direct consequence of the following result, which is e.g.~proved in \cite[Lemma~4.1]{zbMATH07442350} or~\cite[Lemma~4.7]{zbMATH07939044}.

\begin{lemma}\label{lem:monotonesattelite}
If $K$ is obtained from $J$ by a twist, then $P(K)$ is obtained from~$P(J)$ by a twist. \qed
\end{lemma}

Note that the statement of \cref{lem:monotonesattelite} does not hold if one only considers twists on a fixed number of strands (e.g.~only crossing changes).

It follows from \eqref{eq:ypreserves} and \eqref{eq:Ppreserves} that the composition $$y\circ P\colon\CC \to \mathbb{R}$$ is also monotone for every monotone function~$y \colon \CC \to \R$.

We denote by $\CC_+$ the set of those concordance classes $[K]$ in $\CC$ with~$[U] \preceq [K]$, and by~$\CC_-$ the set of those~$[K]$ with $[K] \preceq [U]$
(these sets are called $\widetilde{\mathcal P_0}$ and $\widetilde{\mathcal N_0}$ in~\cite{zbMATH06339370},
and a knot $K$ with $[K]\in\CC_{\pm}$ is called positively or negatively slice in~\cite{zbMATH07880416}). Note that~$\CC_{\pm}$ is closed under connected sum due to the translation-invariance of~$\preceq$.

\begin{theorem}[{Restatement of \Cref{thm:twist_ineq}}]\label{thm:main_detailed}
Let $[K]\in \CC_+$ and let $P$ be a pattern such that
$y(P(K)) \neq y(P(U))$ for some monotone knot concordance invariant~$y$. Then the following hold:
\begin{enumerate}[label=(\alph*)]
\item $K \# J$ is non-slice for all knots $J$ with~$[J]\in \CC_+$. \label{item:KJ_non-slice}
\item $K$ has infinite concordance order. \label{item:inf_order}
\end{enumerate}
These statements also hold after replacing both instances of~$\CC_+$ by~$\CC_-$. 
\end{theorem}

\begin{proof}
We first prove \ref{item:KJ_non-slice} assuming~$[K]\in \CC_+$. Since $y$ and thus $y \circ P\colon\CC\to\mathbb{R}$ are monotone, we have $y(P(U)) \leq y(P(K))$. The assumption $y(P(U)) \neq y(P(K))$ thus implies $y(P(U)) < y(P(K))$. 
Moreover, we have
\begin{align*}
    [U] \preceq [J]
    \Rightarrow 
    [K] \preceq [K \# J]
    \Rightarrow
    y(P(K)) \leq y(P(K\# J))
    \Rightarrow
    y(P(U)) < y(P(K\# J)).
\end{align*}
It follows that $K\# J$ is not slice, as claimed in \ref{item:KJ_non-slice}. Now \ref{item:inf_order} follows from \ref{item:KJ_non-slice} by taking $J = K^{\#n}$ for any~$n \geq 0$.
The proof for the case $\CC_-$ works similarly. 
\end{proof}

\begin{remark}
    When applied to any of the Levine--Tristram signatures~$y=-\sigma_\omega$ for $\omega \in S^1 \setminus \{1\}$ with $\Delta_K(\omega) \neq 0$, \Cref{thm:main_detailed} does not provide any new insights on concordance orders. Indeed, the satellite formula for $\sigma_\omega$ \cite[Theorem~2]{litherland} implies that~$\sigma_{\omega^m}(K) \neq 0$ if $\sigma_\omega(P(K)) \neq \sigma_\omega(P(U))$ for some pattern $P$ with winding number~$m$, in which case computing $\sigma_{\omega^m}(K)$ already shows that $K$ has infinite concordance order. 
\end{remark}

\begin{remark}\label{rmk:tau}
    Let us apply \Cref{thm:main_detailed} for~$y = \tau$. Let $P_{\pm}$ be the $(2,\pm 1)$-cable pattern, which satisfies~$P_{\pm}(U) = U$. Then, for the $\{-1,0,1\}$-valued invariant~$\varepsilon$, we have~\cite[Theorem~1]{zbMATH06366498}
    \[
    \tau(P_{\pm}(K)) =
    \left\{\begin{array}{ll}
    2\tau(K) \pm 1 & \text{if } \varepsilon(K) = \mp 1, \\
    2\tau(K)  & \text{if }\varepsilon(K) \in \{\pm 1, 0\}.
    \end{array}\right.
    \]
    Thus we find $\varepsilon(K) = \mp 1 \Rightarrow \tau(P_{\pm}(K)) \neq \tau(P_{\pm}(U)) = 0$, and so \Cref{thm:main_detailed} does provide a stronger obstruction to 
    finite concordance order than just~$\tau$: namely that $K$ has infinite concordance order if $\varepsilon(K) \neq 0$---but this is already well-known~\cite[Proposition~3.6 (6)]{zbMATH06366498}.
    Interestingly, one does not obtain more information for $y=\tau$ by studying further patterns.
    Indeed, if~$\varepsilon(K) = 0$, then $\tau(K) = 0$~\cite{zbMATH06366498}
    and $\tau(P(K)) = \tau(P(U))$ for all patterns $P$
    (see \cite[Theorem 5]{zbMATH06361412}, \cite[Theorem~2.30]{zbMATH07939044}).
    In other words, we find
    \begin{multline*}
    \tau(P_+(K)) = \tau(P_+(U))\text{ and }
    \tau(P_-(K)) = \tau(P_-(U)) \\
    \Longleftrightarrow\qquad
    \tau(P(K)) = \tau(P(U)) \text{ for all patterns } P.
    \end{multline*}
    The analogous 
    statement does not hold for the Rasmussen invariants, see \cref{rmk:s}.
\end{remark}

In \Cref{sec:picc_patterns}, we will apply \Cref{thm:main_detailed} to specific patterns that result from a generalisation of one of Piccirillo's constructions from~\cite{piccirillo:shake-genus,piccirillo:Conway}. For a given knot~$K$, the associated pattern $P_K$ will have the property that 
\begin{equation}\label{eq:trace}
    K \#J \text{ and } P_K(J) \text{ have the same $0$-trace for every knot~$J$.}
\end{equation}
Recall that the $0$-trace $X_0(K)$ of a knot~$K$ is the oriented, connected, compact, smooth $4$-manifold obtained by attaching a $0$-framed $4$-dimensional $2$-handle to the $4$-ball $B^4$ with attaching sphere~$K$. Property~\eqref{eq:trace} is crucial to proving \Cref{cor:inf_order_pattern} below, which we will then use in \Cref{sec:picc_patterns} to prove \Cref{cor:inf_order} from the introduction. 

In the proof of \Cref{cor:inf_order_pattern}, we also make use of the \emph{trace embedding lemma}. This folklore result states that a knot $K$ is slice if and only if its $0$-trace~$X_0(K)$ can be smoothly embedded in~$S^4$; see \cite[Lemma~1.3]{piccirillo:Conway}, \cite[Theorem~1.8]{miller-piccirillo} and~\cite{kirby-melvin}. Consequently, two knots $K$ and $K\pr$ with the same $0$-trace $X_0(K) = X_0(K\pr)$ have the same sliceness status. The trace embedding lemma was famously employed in~\cite{piccirillo:Conway} to prove that the Conway knot is not slice. 

\begin{prop}\label{cor:inf_order_pattern}
Let $[K]\in \CC_+$ and let $P_K$ be a pattern satisfying~\eqref{eq:trace}. Suppose that~$y(P_K(U)) \neq 0$ for a monotone knot concordance invariant~$y$. Then~$K \# J$ is non-slice for all knots $J$ with~$[J]\in \CC_{+}$. In particular, $K$ has infinite concordance order. These statements also hold after replacing both instances of~$\CC_+$ by~$\CC_-$. 
\end{prop}

\begin{proof}
Using the trace embedding lemma, the assumption~\eqref{eq:trace} implies that~$P_K(-K)$ is slice, since $K \# {-K}$ is slice. We thus have $y(P_K(-K)) = 0 \neq y(P_K(U))$. Since~$[K]\in \CC_+$, we have $[-K]\in \CC_-$. \Cref{thm:main_detailed} now implies that $-K \# J$ is non-slice for $[J]\in \CC_{-}$. This implies that $K\# J$ is non-slice for $[J]\in \CC_{+}$, as claimed.
\end{proof}

\section{Piccirillo patterns}\label{sec:picc_patterns}

In this section, we create the pattern $P_{K,c}$ associated with an unknotting crossing~$c$ of a knot $K$ with unknotting number~$1$, as promised 
in the introduction, and use it to show that the Conway knot has infinite concordance order. This pattern will arise from an RBG link construction, which is a general method of creating knots that share a trace or a surgery. There are several slightly different versions of this construction, where the main idea is always to describe a surgery diagram or Kirby diagram of a manifold using a $3$-component link such that two pairs of these components cancel each other out. See for example~\cite{Akbulut,BM_non_char,piccirillo:shake-genus,piccirillo:Conway,manolescu-piccirillo:RBG,Tagami,qin,KP_four_surgeries,RBG_high_dim,HPW}. Here, we will use the following variant, which is a special case of Definition~1.1 in~\cite{manolescu-piccirillo:RBG}.

\begin{definition}
An oriented 3-component link $L=R\cup B\cup G$ in $S^3$ is called an \emph{RBG link} if
\begin{enumerate}[label=(\alph*)]
    \item $R\cup B$ is isotopic to $\mu_B\cup B$; 
    \item $R\cup G$ is isotopic to $\mu_G \cup G$,
    where $\mu_B$ and $\mu_G$ denote meridians of $B$ and~$G$, respectively; and
    \item $R$ is labelled with a dot, and $B$ and $G$ are labelled with integers $b$ and~$g$, such that $b+g=2\lk(B,G)$.  
\end{enumerate}
\end{definition}

Let $L=R\cup B\cup G$ be an RBG link. Since $R$ is an unknot labelled with a dot, we can interpret $L$ as a Kirby diagram of the $4$-manifold $X_L$ obtained from $B^4$ by attaching a $1$-handle along $R$ (or more precisely by pushing a Seifert disk $D_R$ of $R$ into the $4$-ball and removing an open tubular neighbourhood of it from~$B^4$), and attaching $2$-handles along $B$ and $G$ with framings $b$ and~$g$, respectively. See~\cite[Chapters~4 and~5]{gompf-stipsicz:book} for an introduction to Kirby diagrams and Kirby calculus.
We write $B^4\cup h_1(R)\cup h_2(B)$ for the $4$-manifold obtained by just attaching the $1$-handle and the $2$-handle along $B$ to~$B^4$. By assumption, there exists an isotopy such that~$B$ intersects $D_R$ transversely in a single point. Thus, the $1$- and the $2$-handle cancel and there exists a diffeomorphism 
$$f_B\colon B^4\cup h_1(R)\cup h_2(B)\rightarrow B^4.$$ 
We denote by $K_G$ the image of $G$ under this diffeomorphism, i.e.~$f_B(G)=K_G$. By interchanging the roles of $B$ and~$G$, we also obtain a knot~$K_B$.
By construction, the knots $K_B$ and $K_G$ share the same $0$-trace, as summarised in the next proposition.

\begin{prop}\label{thm:K_B_and_K_G}
Let $L=R\cup B\cup G$ be an RBG link. Then $K_B$ and $K_G$ have orientation-preservingly diffeomorphic $0$-traces.
\end{prop}

\begin{proof}
We write $f_B(g)$ for the image of the framing~$g$ of $G$ under the cancellation diffeomorphism~$f_B$. Then $X_L$ is diffeomorphic to the $(f_B(g))$-trace of~$K_G$. Similarly, the $(f_G(b))$-trace of $K_B$ is diffeomorphic to $X_L$ and thus to the $(f_B(g))$-trace of~$K_G$.
An explicit computation using handle slides in Kirby calculus, as shown in \Cref{fig:main_figure}, reveals that $f_B(g)=f_G(b)=b+g-2 \lk(B,G)=0$ as claimed.
\end{proof}

Note that \Cref{thm:K_B_and_K_G} does not state that $K_B$ and $K_G$ are necessarily distinct. In fact, if the initial RBG link is overly simple or symmetric, the two knots might be isotopic. However, for sufficiently complicated RBG links, there is no reason to expect them to be isotopic.

\begin{theorem}\label{thm:satellite_knots}
    Let $K$ be a knot with unknotting number~$1$ and an 
    unknotting crossing~$c$. 
    Then there are 
    two patterns $P_{K,c}$ and $Q_{K,c}$ such that
    \begin{enumerate}[label=(\alph*)]
        \item $Q_{K,c}(U)=K$,
        \item for any two knots $J$ and~$M$, the knots 
        $Q_{K,c}(M) \#J$ and $P_{K,c}(J) \# M$ have the same $0$-trace.
    \end{enumerate}
    In particular, setting $J=M=U$ to be the unknot, the knots $Q_{K,c}(U)=K$ and~$P_{K,c}(U)$ have the same $0$-trace. 
\end{theorem}

We refer to $P_{K,c}$ as the \emph{Piccirillo pattern} of $K$ at the unknotting crossing~$c$.

\begin{remark}
\label{rmk:picc}
    \begin{enumerate}[label=(\alph*)]
        \item The knots from \Cref{thm:satellite_knots} will arise as $K_B$ and $K_G$ from a particular RBG link~$L = R \cup B \cup G$.
        In fact, only part (a) of \Cref{thm:satellite_knots} requires the unknotting number~$1$ 
        hypothesis. Part (b) of the theorem holds for any RBG link. Namely, given an RBG link~$L=R \cup B \cup G$, 
        there exist two associated patterns $P$ and $Q$ such that $Q(M)\# J$ and $P(J)\# M$ have the same $0$-trace for any two knots~$M$ and~$J$.
	    However, our main motivation here is to construct a pattern $P_K$ with property~\eqref{eq:trace},
	    which can then be used in \cref{cor:inf_order_pattern}.
	    Hence, we only present the construction in this special case.
        \item The RBG link that we construct in the proof of \Cref{thm:satellite_knots} is similar to a link that arises in Piccirillo's construction 
        in Proposition~2.2 of~\cite{piccirillo:Conway}; see also the middle frame of Figure~3 therein.
        Indeed, the $J=M=U$ case of \Cref{thm:satellite_knots} recovers Piccirillo's construction. In this case, we obtain two knots $K_B = Q_{K,c}(U)=K$ and~$K_G=P_{K,c}(U)$ which have the same $0$-trace. When~$K=C$, the Conway knot, the knot $K_G$ is the knot called~$K\pr$ in \cite{piccirillo:Conway}, 
        whose non-vanishing $s$ invariant was used there to prove that $C$ is not slice.
        \item The patterns~$Q_{K,c}$ and~$P_{K,c}$ correspond to the dualisable pattern $P$ and its dual pattern $P^\ast$ from \cite[Proposition~4.2]{piccirillo:shake-genus}. 
        \item Once again, in the statement of \Cref{thm:satellite_knots} we are not claiming that the patterns~$Q_{K,c}$ and $P_{K,c}$ or the knots $Q_{K,c}(M) \#J$ and $P_{K,c}(J) \# M$ are necessarily distinct. In fact, if we take for both~$Q_{K,c}$ and $P_{K,c}$ the pattern that performs a connected sum with~$K$, then \Cref{thm:satellite_knots} is trivially true. However, the patterns that we construct in the proof of \Cref{thm:satellite_knots} below are generally more interesting. Even for~$J=M=U$, the corresponding knots are usually non-isotopic; see for example \cite[Conjecture~5.5]{KS_exotic_traces} and the supporting data.
    \end{enumerate}
\end{remark}

    \begin{figure}[t]
	   \centering
	   \def\svgwidth{0.8\columnwidth}
	   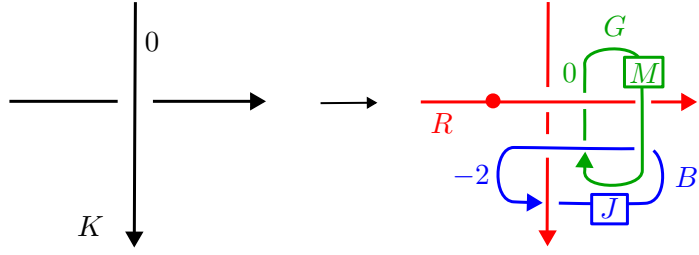
       \caption{Left: a positive unknotting crossing $c$ of~$K$. Right: an RBG link~$L = R \cup B \cup G$. Outside the shown tangle, the red knot~$R$ coincides with~$K$. The boxes labelled $M$ and $J$
       represent the connected sums of the green and blue  
       components with the knots~$M$ and~$J$, respectively.
       }
	   \label{fig:RBG}
    \end{figure}

\begin{proof}[{Proof of \Cref{thm:satellite_knots}}]
    Let $c$ be an unknotting crossing of~$K$. Consider first the case that $c$ is a positive crossing, as shown on the left in \Cref{fig:RBG}. The $3$-component link $L = R \cup B \cup G$ on the right of~\Cref{fig:RBG} with~$b=-2$ and $g=0$ is an RBG link. Indeed, since the shown crossing is an unknotting crossing, the red knot~$R$ is an unknot, and thus a meridian of $B$ and~$G$. 
    We have $\lk(B,G) = -1$, and hence~$b+g = 2\lk(B,G)$.
    If the unknotting crossing~$c$ is negative, the construction of the associated RBG link is similar. However, in this case, the vertical strands of both~$K$ and~$R$ in~\Cref{fig:RBG} are oriented upwards, and the blue component~$B$ of~$L$, which in this case has framing~$b=2$, has the opposite orientation. Otherwise, the RBG link is the same. Consequently, in the case of a negative crossing, 
    we obtain~$\lk(B,G)=1$, and therefore $b+g = 2\lk(B,G)$ again.

        \begin{figure}[t]
    \centering
    \tiny   
    \def\svgwidth{1\columnwidth}  
    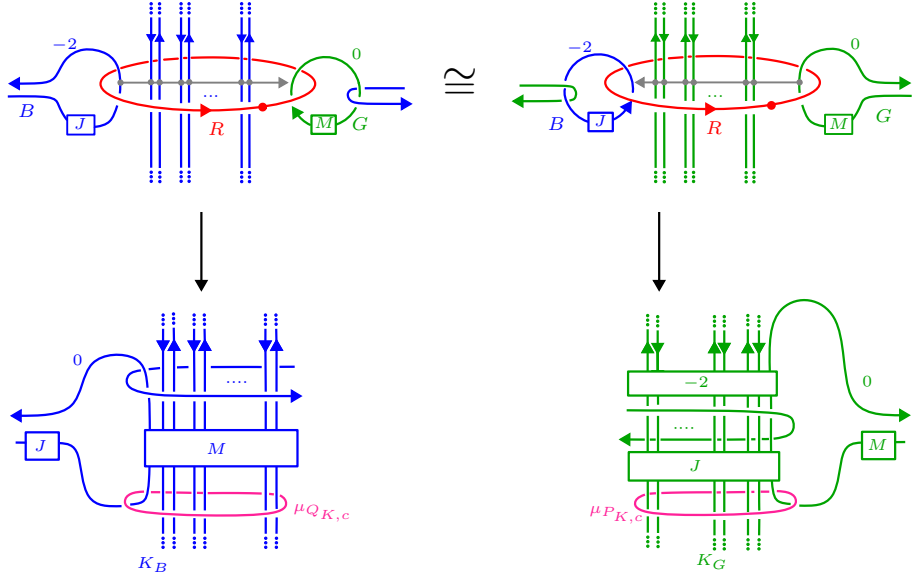
    \caption{Top row: two views of the same RBG link~$L$. Bottom row: the knots $K_B$ and $K_G$ have the same $0$-trace. For the knot~$K_B$, the box labelled $M$ represents both the diagram of the knot~$M$ and the ($-w(M)$)-full twists produced by the handle-sliding operation, where $w(M)$ denotes the writhe of the diagram of~$M$; while the box labelled $J$ indicates the connected sum with~$J$. For~$K_G$, the roles of~$J$ and $M$ are exchanged. The box labelled $-2$ represents two left-handed full twists between the indicated strands.}
    \label{fig:main_figure}
\end{figure}

    Since $R$ is a meridian of both $G$ and~$B$, the diagram of $L$ can be isotoped into the top-left or top-right configurations shown in Figure~\ref{fig:main_figure}. Here, we depict
    again only the case where the unknotting crossing $c$ is positive; the case where the crossing is negative works analogously.
    Performing the handle slides indicated by the grey arrows in Figure~\ref{fig:main_figure} and cancelling the $1$- and $2$-handle pairs~$(R, G)$ and~$(R, B)$, respectively, yields the knots~$K_B$ and~$K_G$, as shown in the bottom row of Figure~\ref{fig:main_figure}. By \Cref{thm:K_B_and_K_G} (or directly by the Kirby moves just described), the knots~$K_B$ and~$K_G$ have the same $0$-trace.

    As indicated in \Cref{fig:main_figure}, the knot $K_B$ can be viewed as the result of applying the satellite operation with pattern
    $Q_{K,c}$ and companion~$M$, and then taking the connected sum with the knot~$J$. An analogous description holds for~$K_G$, using the pattern~$P_{K,c}$. In particular, we have
    \begin{equation*}
        K_B = Q_{K,c}(M) \#J \quad \text{and} \quad K_G = P_{K,c}(J) \# M.
    \end{equation*}
    Here, the patterns $Q_{K,c}$ and $P_{K,c}$ are the blue and green knots shown in the bottom row of \Cref{fig:main_figure}, setting $J = M =U$ and viewing them in the solid tori determined by the meridians $\mu_{Q_{K,c}}$ and $\mu_{P_{K,c}}$ (in pink), respectively.    

Finally, we show that~$Q_{K,c}(U)=K$.
To do so,
we set~$J=M=U$. By construction, this implies that $Q_{K,c}(U)$ is the image of the blue component $B$ after cancelling the $1$- and $2$-handle pair~$(R,G)$.
To identify this knot, we consider instead the Kirby diagram in \Cref{fig:RBG} with~$J=M=U$, and surger the red dotted circle~$R$, replacing the dot by a $0$-framing, as shown in the left-hand frame of \Cref{fig:isotopy}. Note that this changes the $4$-manifold, but not its boundary $3$-manifold. 
In the resulting diagram, slide the blue component $B$ over the $0$-framed red component $R$ along the grey arrow in \Cref{fig:isotopy}, ensuring that the green component $G$ becomes a $0$-framed meridian of $R$ that does not link the new blue component.
The pair $(R,G)$ can then be removed using a slam-dunk move (see, for example, \cite[p.~163]{gompf-stipsicz:book}), after which
the image of the blue component is precisely the knot $K$ with framing~$0$, as shown in the right-hand frame of \Cref{fig:isotopy}. 

This identifies $K$ with the image of $B$ under the diffeomorphism
$$
f\colon\partial\bigl(B^4 \cup h_2(R)\cup h_2(G)\bigr)\to S^3,
$$
induced by surgery on the red $1$-handle, a $2$-handle slide,
and the slam-dunk move.

On the other hand, the knot $K_B=Q_{K,c}(U)$ was defined as the image of $B$ under the cancellation diffeomorphism
$$
f_G \colon B^4\cup h_1(R)\cup h_2(G)\to B^4.
$$
Therefore,
the composition $f_G|_{\partial}\circ f^{-1}$ is a diffeomorphism of $S^3$ that maps $K$ to $K_B=Q_{K,c}(U)$. Since equivalent knots in $S^3$ are isotopic, it follows that $Q_{K,c}(U)=K$, as claimed.

Note that, alternatively, we could stick with the red component $R$ as a %
dotted $1$-handle in \Cref{fig:RBG} and use Kirby calculus to deduce that~$Q_{K,c}(U)=K$,  
similar to the argument used in \cite[Proposition~2.2]{piccirillo:Conway}.
\end{proof}

 \begin{figure}[t]
    \centering
    \def\svgwidth{1.\columnwidth}    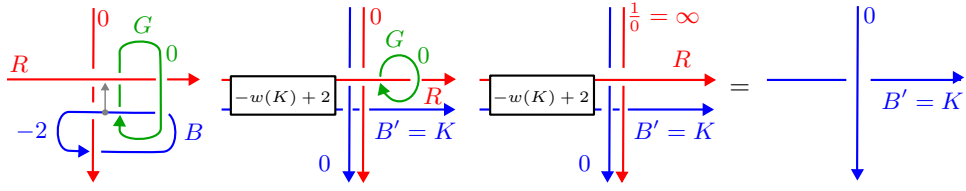
    \caption{From left to right: the surgery diagram obtained by replacing the red dotted circle $R$ with a $0$-framed $2$-handle, the result of sliding the $(-2)$-framed component $B$ over the $0$-framed component~$R$, and the result of a slam-dunk move on~$(R,G)$. 
    }
    \label{fig:isotopy}
\end{figure}

We are now ready to prove \Cref{cor:inf_order} from the introduction. 

\begin{proof}[{Proof of \Cref{cor:inf_order}}]
    Let $K$ be a knot with unknotting number~$1$ and an unknotting crossing~$c$. By \Cref{thm:satellite_knots} with~$M=U$, the knots $K \# J$ and $P_{K,c}(J)$ have the same $0$-trace for every knot~$J$. Note that the sign $\operatorname{sgn}(c)$ of the crossing~$c$ determines whether $[K]$ belongs to $\CC_+$ or~$\CC_-$. If $s(P_{K,c}(U))\neq 0$, then \Cref{cor:inf_order_pattern} implies that~$K \# J$ is non-slice for all knots $J$ with $[J]\in \CC_{\operatorname{sgn}(c)}$. Here we use the fact that $s$ is monotone~\cite[Theorem~1.11]{manolescu-marengon-sarkar-willis}.
\end{proof}

Note that in the statement of \Cref{cor:inf_order}, we could replace the Rasmussen $s$ invariant with any monotone knot concordance invariant~$y$.

Finally, we apply our method to the Conway knot.

\begin{proof}[{Proof of \Cref{thm:Conway_inf_order}}]
    From the unknotting crossing $c$ of the Conway knot $C$ shown in Figure~\ref{fig:conway_crossing}, we obtain the pattern $P_{C,c}$ from \Cref{thm:satellite_knots}. By \Cref{rmk:picc}(b), 
    the satellite knot~$P_{C,c}(U)$ is isotopic to the knot called  $K'$ in \cite{piccirillo:Conway}, whose $s$ invariant was computed there to be non-zero.
    Thus, the theorem follows from \Cref{cor:inf_order}.
\end{proof}

\begin{figure}[ht]
     \def\svgwidth{0.3\columnwidth}
     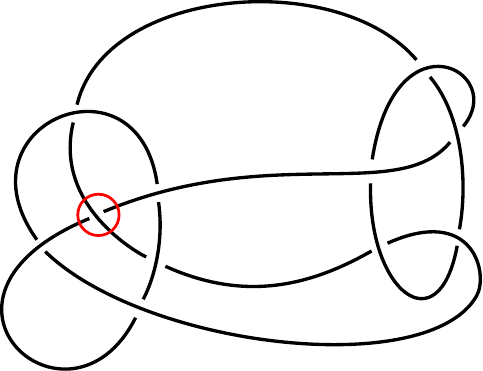
     \caption{An unknotting crossing $c$ of the Conway knot~$C$. 
     }
     \label{fig:conway_crossing}
\end{figure}

\begin{remark}\label{rem:generalised_Conway_knots}
We now explain how our methods give an alternative proof of a result of Golla--Pinzón-Caicedo~\cite[Theorem~1.3]{golla_pinzon-caicedo}, namely that the generalised Conway knot $C_{2,n}$ has infinite concordance order for every $n\neq 0$.

For the generalised Conway knots $C_{r,n}$ (see \cite[Figure~2]{golla_pinzon-caicedo}), a positive-to-negative crossing change transforms $C_{r,n}$ into $C_{r,n+1}$. Hence
\begin{equation*}
[C_{r,n+1}] \preceq [C_{r,n}] \preceq \cdots \preceq [C_{r,1}] \preceq [C_{r,0}]=[U].
\end{equation*}
For $r=2$, the knot $C_{2,1}$ is the inverse of the Conway knot with our conventions; see the introduction. Thus, for every $n\geq 1$,
\begin{equation*}
[C_{2,n}] \preceq \cdots \preceq [C_{2,1}]=[-C]\preceq [C_{2,0}]=[U].
\end{equation*}
Applying $s\circ P_{C,c}$ gives
\begin{equation*}
s(P_{C,c}(C_{2,n}))
\leq s(P_{C,c}(-C))
=0
<s(P_{C,c}(U))
=s(K')
=2.
\end{equation*}
It therefore follows from \Cref{thm:twist_ineq} that $C_{2,n}$ has infinite concordance order for every $n\geq 1$. Since $C_{2,-n}$ is the mirror of $C_{2,n}$, the same conclusion holds for every~$n\neq 0$.
\end{remark}

\section{Experimental calculations}\label{sec:experiments}

In this section, we investigate the application of \cref{thm:main_detailed} to prime knots with a low crossing number as a means of establishing infinite concordance order. More details can be found in~\cite{data}.

On the one hand, if $[K]\in\CC_{\pm}$ and $y(K) \neq y(U)$ for some monotone knot concordance invariant~$y$, then \cref{thm:main_detailed} may be applied with $P$ the trivial pattern---but of course, in that case the infinite concordance order of $K$ also follows directly from~$y(K) \neq y(U)$. On the other hand, if $[K] \in \CC_+ \cap \CC_-$ (such knots were named \emph{BPH-slice} by Manolescu--Piccirillo~\cite{manolescu-piccirillo:RBG}), then $y(P(K)) = y(P(U))$ for all patterns~$P$, and so there is no hope of applying \cref{thm:main_detailed} (this was also observed in~\cite[Theorem~A]{zbMATH06696649}). We are thus interested in knots $K$ with the following properties: 
\begin{enumerate}[label=(\alph*)]
\item There is no known monotone concordance invariant $y$ with~$y(K) \neq 0$.
\item The concordance order of $K$ is infinite, or unknown.
\item $[K] \not\in \CC_+ \cap \CC_-$.
\item $[K] \in \CC_+\cup \CC_-$.
\end{enumerate}
Among the 2977 prime knots with crossing number 12 or less, there are 480 knots satisfying (a) and (b)~\cite{knotinfo}. For condition~(a), we checked that the Rasmussen~$s$ invariant (which equals $2\tau$ for knots with crossing number 12 or less) and the Levine--Tristram signatures of $K$ vanish. There are no known algorithms for deciding whether a given concordance class lies in~$\CC_{\pm}$. So we rely on computer experiments, continuing the efforts of Manolescu--Piccirillo~\cite{manolescu-piccirillo:RBG}. Among the aforementioned~480 knots, we find that 337 are in~$\CC_+ \cap \CC_-$, and for the 143 remaining knots, we do not know. Among those $143$ knots, we find that 138 are in~$\CC_+ \cup \CC_-$, and for the following 5 knots, we do not know if they are in~$\CC_+ \cup \CC_-$:
\[
K12a899,\, K12a917,\, K12a1180,\, K12a1222,\, K12n880.
\]
So those 138 knots (see \cref{tab:candidates}) are the candidates for which \cref{thm:main_detailed} may establish infinite concordance order.

\begin{table}[htbp]
\begin{tabular}{>{\raggedleft}p{2cm}|r|>{\raggedleft\arraybackslash}p{2cm}}
Cr.~number / Alternating
         & (a), (b)
               & (a), (b), (d); (c)~possible \\\hline
$\leq 9$ &   9 &  0    \\
10a      &  19 &  6    \\
10n      &   4 &  0    \\
11a      &  65 & 19    \\ 
11n      &  24 &  4    \\ 
12a      & 221 & 75    \\
12n      & 138 & 34    \\
\end{tabular}
\bigskip

\caption{The 138 candidate knots, broken down by crossing number and alternatingness.}
\label{tab:candidates}
\end{table}

The only 12-crossing knot with unknown concordance order, K12n846, is unfortunately BPH-slice, so it does not satisfy~(c).
Among the 138 candidates, there are the following 14 knots with known unknotting number~$1$:
\begin{multline*}
K10a118, \, K11n34,\, K11n153,\, K12a261,\, K12a303,\, K12a635,\, K12a865,\\
K12a1048,\, K12a1235,\, K12a1239,\, K12a1250,\, K12n31,\, K12n300,\, K12n816.
\end{multline*}
For these knots, we try to apply \cref{thm:main_detailed} with a Piccirillo pattern. This works for the Conway knot $K11n34$ as discussed in the previous section, and moreover for the knot~$K = K12n31$, for which we compute the $\mathbb{F}_3$-Rasmussen invariant to be~${s_3(P_{K,c}(U)) = 2}$. By \cref{cor:inf_order}, this recovers the fact that $K$ has infinite concordance order, which could previously be shown using twisted Alexander polynomials~\cite{collins_thesis}. 

Unfortunately, for the remaining 12 candidates $K$ with~$u(K) = 1$, the Rasmussen invariants of $P_{K,c}(U)$ appear to be beyond the scope of computer calculations due to the high crossing number of the diagrams we found for~$P_{K,c}(U)$.

If we drop the condition~$u(K) = 1$, we can also try applying \cref{thm:main_detailed} using patterns other than Piccirillo patterns. Our calculations are ongoing, but have so far not met with success. One reason for this is that patterns $P$ without property~\eqref{eq:trace} have the disadvantage, from a practical viewpoint, that $y(P(K))$ needs to be calculated. If $K$ has crossing number $c$ and $P$ has wrapping number $w$ (also known as \emph{geometric winding number}), the diagram one naively finds for $P(K)$ will have a crossing number of at least~$cw^2$. Since our candidate knots have~$c\geq 10$, this means that calculations of Rasmussen invariants appear to be feasible only for~$w = 2$.

\begin{remark}\label{rmk:s}
Among the simplest choices for $P$ are the $(2,\pm 1)$-cable patterns, denoted~$P_{\pm}$. We remind the reader of \cref{rmk:tau}: if ${\tau(P_+(K)) = \tau(P_-(K)) = 0}$, then $\tau(P(K)) = \tau(P(U))$ for all patterns~$P$. It was conjectured in~\cite[Conjecture~2.31]{zbMATH07939044} that this should also hold for the Rasmussen invariants $s_p$ over a field of any characteristic~$p$. However, the Conway knot $C$ provides a counterexample: one can verify that $s_p(P_+(C)) = s_p(P_-(C)) = 0$, but $0 = s_p(P(C)) \neq s_p(P(U))$ for the pattern $P = -P_{C,c}$ and $p = 3,5,7$
(though~$s_2(P(U))=0$,  
so the conjecture remains open for $p = 2$).
Here, the pattern $P = -P_{C,c}$ has wrapping number~$5$.

Nevertheless, a weaker statement holds for~$s_p$: if $s_p(P_+(K)) = s_p(P_-(K)) = 0$, then $s_p(P(K)) = s_p(P(U))$ for $P$ a pattern with wrapping number~$2$,
as follows from \cite{arXiv:2602.23442} (see also~\cite{zbMATH07939044,zbMATH08091067}).
\end{remark}

\bibliographystyle{myamsalpha}
\bibliography{main}

@article{manolescu-marengon-sarkar-willis,
 author = {Manolescu, Ciprian and Marengon, Marco and Sarkar, Sucharit and Willis, Michael},
 title = {A generalization of {Rasmussen}'s invariant, with applications to surfaces in some four-manifolds},
 fjournal = {Duke Mathematical Journal},
 journal = {Duke Math. J.},
 issn = {0012-7094},
 volume = {172},
 number = {2},
 pages = {231--311},
 year = {2023},
 doi = {10.1215/00127094-2022-0039},
 url = {real.mtak.hu/163239/1/1910.08195v2.pdf},
 zbMATH = {7653256},
 Zbl = {1535.57014}
}

@misc{Baldwin2025ConwayKnots,
  author       = {Baldwin, John A.},
  title        = {Are any of the {C}onway knots slice?},
  howpublished = {Problems in Floer Homology, \url{https://floerhomologyproblems.blogspot.com/2025/09/are-any-of-conway-knots-slice.html}},
  year = {2025},
  note         = {Blog post, September 7, 2025},
}

@article{Table,
 author = {Hoste, Jim and Thistlethwaite, Morwen and Weeks, Jeff},
 title = {The first 1,701,936 knots},
 fjournal = {The Mathematical Intelligencer},
 journal = {Math. Intell.},
 issn = {0343-6993},
 volume = {20},
 number = {4},
 pages = {33--48},
 year = {1998},
 doi = {10.1007/BF03025227},
 zbMATH = {1279095},
 Zbl = {0916.57008}
}

@article{DTnotation,
 author = {Dowker, C. H. and Thistlethwaite, Morwen B.},
 title = {Classification of knot projections},
 fjournal = {Topology and its Applications},
 journal = {Topology Appl.},
 issn = {0166-8641},
 volume = {16},
 pages = {19--31},
 year = {1983},
 doi = {10.1016/0166-8641(83)90004-4},
 zbMATH = {3816622},
 Zbl = {0516.57002}
}

@misc{data,
      title={Data accompanying this paper},
      author={Donatone, Chiara and Kegel, Marc and Lewark, Lukas and Truöl, Paula},
      year={2026},
      howpublished = {available as ancillary files of this article on the arXiv and on \url{https://github.com/LLewark/conway-knot-has-infinite-concordance-order}}
}

@misc{RBG_high_dim,
      title={On the detection of knotted spheres by their traces in high dimensions}, 
      author={Valentina Bais and Alessio {Di Prisa} and Daniel Hartman and Chun-Sheng Hsueh and Marc Kegel and Alice Merz and Mark Pencovitch and Arunima Ray and Diego Santoro and Paula Truöl and Laura Wakelin},
      year={2025},
      arxiv={2511.07251}, 
      howpublished = {Preprint},
}

@misc{HPW,
      title={Dehn surgery functions are never injective}, 
      author={Kyle Hayden and Lisa Piccirillo and Laura Wakelin},
      year={2025},
      arxiv={2508.13369}, 
      howpublished = {Preprint},
}

@article{Akbulut,
 author = {Akbulut, Selman},
 title = {A fake compact contractible 4-manifold},
 fjournal = {Journal of Differential Geometry},
 journal = {J. Differ. Geom.},
 issn = {0022-040X},
 volume = {33},
 number = {2},
 pages = {335--356},
 year = {1991},
 doi = {10.4310/jdg/1214446320},
 zbMATH = {11236},
 Zbl = {0839.57015}
}

@article{BM_non_char,
 author = {Baker, Kenneth L. and Motegi, Kimihiko},
 title = {Noncharacterizing slopes for hyperbolic knots},
 fjournal = {Algebraic \& Geometric Topology},
 journal = {Algebr. Geom. Topol.},
 issn = {1472-2747},
 volume = {18},
 number = {3},
 pages = {1461--1480},
 year = {2018},
 doi = {10.2140/agt.2018.18.1461},
 zbMATH = {6866404},
 Zbl = {1422.57010}
}

@misc{KP_four_surgeries,
      title={Knots that share four surgeries}, 
      author={Marc Kegel and Lisa Piccirillo},
      year={2025},
      arxiv={2505.13168}, 
      howpublished = {Preprint},
}

@misc{KS_exotic_traces,
      title={The search for exotic knot traces}, 
      author={Marc Kegel and Jonathan Spreer},
      year={2026},
      arxiv={2603.22438}, 
      howpublished = {Preprint},
}

@article{zbMATH06272212,
 author = {Kronheimer, P. B. and Mrowka, T. S.},
 title = {Gauge theory and {Rasmussen}'s invariant},
 fjournal = {Journal of Topology},
 journal = {J. Topol.},
 issn = {1753-8416},
 volume = {6},
 number = {3},
 pages = {659--674},
 year = {2013},
 doi = {10.1112/jtopol/jtt008},
 zbMATH = {6272212},
 Zbl = {1298.57025}
}

@article{zbMATH02057402,
 author = {Ozsv{\'a}th, Peter and Szab{\'o}, Zolt{\'a}n},
 title = {Knot {Floer} homology and the four-ball genus},
 fjournal = {Geometry \& Topology},
 journal = {Geom. Topol.},
 issn = {1465-3060},
 volume = {7},
 pages = {615--639},
 year = {2003},
 doi = {10.2140/gt.2003.7.615},
 url = {https://eudml.org/doc/123536},
 zbMATH = {2057402},
 Zbl = {1037.57027}
}

@article{zbMATH04096384,
 author = {Cochran, Tim D. and Gompf, Robert E.},
 title = {Applications of {Donaldson}'s theorems to classical knot concordance, homology 3-spheres and property {P}},
 fjournal = {Topology},
 journal = {Topology},
 issn = {0040-9383},
 volume = {27},
 number = {4},
 pages = {495--512},
 year = {1988},
 doi = {10.1016/0040-9383(88)90028-6},
 zbMATH = {4096384},
 Zbl = {0669.57003}
}

@article{zbMATH06198082,
 author = {Cochran, Tim D. and Harvey, Shelly and Horn, Peter},
 title = {Filtering smooth concordance classes of topologically slice knots},
 fjournal = {Geometry \& Topology},
 journal = {Geom. Topol.},
 issn = {1465-3060},
 volume = {17},
 number = {4},
 pages = {2103--2162},
 year = {2013},
 doi = {10.2140/gt.2013.17.2103},
 zbMATH = {6198082},
 Zbl = {1282.57006}
}

@article{piccirillo:Conway,
author = {Piccirillo, Lisa},
title = {{The Conway knot is not slice}},
volume = {191},
 journal = {Ann. Math. (2)},
number = {2},
publisher = {Department of Mathematics of Princeton University},
pages = {581--591},
year = {2020},
doi = {10.4007/annals.2020.191.2.5},
URL = {https://doi.org/10.4007/annals.2020.191.2.5},
zbl = {1471.57011},
mrnumber = {4076631},
}

@article{rasmussen,
 author = {Rasmussen, Jacob},
 title = {Khovanov homology and the slice genus},
 fjournal = {Inventiones Mathematicae},
 journal = {Invent. Math.},
 issn = {0020-9910},
 volume = {182},
 number = {2},
 pages = {419--447},
 year = {2010},
 doi = {10.1007/s00222-010-0275-6},
 zbMATH = {5818344},
 Zbl = {1211.57009}
}

@inproceedings{litherland,
	author = {Litherland, R. A.},
	booktitle = {Topology of low-dimensional manifolds ({P}roc. {S}econd {S}ussex {C}onf., {C}helwood {G}ate, 1977)},
	mrclass = {57M25},
	mrnumber = {547456},
	mrreviewer = {Lee Rudolph},
	pages = {71--84},
	publisher = {Springer, Berlin},
	series = {Lecture Notes in Math.},
	title = {Signatures of iterated torus knots},
	volume = {722},
	year = {1979},
 Zbl = {0412.57002}
}

@article{Levine1969,
	author = {Levine, J.},
	journal = {Comment. Math. Helvetici},
	pages = {229--244},
	title = {Knot Cobordism Groups in Codimension Two},
	url = {http://eudml.org/doc/139394},
	volume = {44},
	year = {1969},
 Zbl = {0176.22101},
}

@article{tristram_1969,
	author = {Tristram, A. G.},
	doi = {10.1017/S0305004100044947},
	journal = {Math. Proc. Cambridge Philos. Soc.},
	number = {2},
	pages = {251--264},
	publisher = {Cambridge University Press},
	title = {Some cobordism invariants for links},
	volume = {66},
	year = {1969},
 Zbl = {0191.54703}
}

@article{kirby-melvin,
 author = {Kirby, Robion and Melvin, Paul},
 title = {Slice knots and property {R}},
 fjournal = {Inventiones Mathematicae},
 journal = {Invent. Math.},
 issn = {0020-9910},
 volume = {45},
 pages = {57--59},
 year = {1978},
 doi = {10.1007/BF01406223},
 url = {https://eudml.org/doc/142536},
 zbMATH = {3587731},
 Zbl = {0377.55002}
}

@article{miller-piccirillo,
 author = {Miller, Allison N. and Piccirillo, Lisa},
 title = {Knot traces and concordance},
 fjournal = {Journal of Topology},
 journal = {J. Topol.},
 issn = {1753-8416},
 volume = {11},
 number = {1},
 pages = {201--220},
 year = {2018},
 doi = {10.1112/topo.12054},
 zbMATH = {6864029},
 Zbl = {1393.57010}
}

@article{manolescu-piccirillo:RBG,
 author = {Manolescu, Ciprian and Piccirillo, Lisa},
 title = {From zero surgeries to candidates for exotic definite 4-manifolds},
 fjournal = {Journal of the London Mathematical Society. Second Series},
 journal = {J. Lond. Math. Soc., II. Ser.},
 issn = {0024-6107},
 volume = {108},
 number = {5},
 pages = {2001--2036},
 year = {2023},
 doi = {10.1112/jlms.12800},
 zbMATH = {7780667},
 Zbl = {1541.57006}
}

@article{zbMATH06696649,
 author = {Cha, Jae Choon and Kim, Min Hoon},
 title = {Rasmussen {{\(s\)}}-invariants of satellites do not detect slice knots},
 fjournal = {Journal of Knot Theory and its Ramifications},
 journal = {J. Knot Theory Ramifications},
 issn = {0218-2165},
 volume = {26},
 number = {2},
 pages = {7 p., Id/No 1740001},
 year = {2017},
 doi = {10.1142/S0218216517400016},
 zbMATH = {6696649},
 Zbl = {1361.57004}
}

@misc{knotinfo,
Author = {Livingston, Charles and Moore, Allison H.},
howpublished = {\url{https://knotinfo.org}},
Title = {Knot{I}nfo: Table of {K}not {I}nvariants},
Year = {2026},
month = {retrieved May},
}

@article{zbMATH06339370,
 author = {Cochran, Tim D. and Tweedy, Eamonn},
 title = {Positive links},
 fjournal = {Algebraic \& Geometric Topology},
 journal = {Algebr. Geom. Topol.},
 issn = {1472-2747},
 volume = {14},
 number = {4},
 pages = {2259--2298},
 year = {2014},
 doi = {10.2140/agt.2014.14.379},
 zbMATH = {6339370},
 Zbl = {1311.57008}
}

@article{zbMATH07939044,
 author = {Lewark, Lukas and Zibrowius, Claudius},
 title = {Rasmussen invariants of {Whitehead} doubles and other satellites},
 fjournal = {Journal f{\"u}r die Reine und Angewandte Mathematik},
 journal = {J. Reine Angew. Math.},
 issn = {0075-4102},
 volume = {816},
 pages = {241--296},
 year = {2024},
 doi = {10.1515/crelle-2024-0061},
 zbMATH = {7939044},
 Zbl = {1565.57011}
}

@book{zbMATH00877622,
 author = {Rolfsen, Dale},
 title = {Knots and links.},
 edition = {2nd print. with corr.},
 fseries = {Mathematics Lecture Series},
 series = {Math. Lect. Ser.},
 volume = {7},
 isbn = {0-914098-16-0},
 year = {1990},
 publisher = {Houston, TX: Publish or Perish},
 zbMATH = {877622},
 Zbl = {0854.57002}
}

@article{ren2025adjunctioninequalityspatiallyrefined,
 author = {Ren, Qiuyu},
 title = {Adjunction inequality for spatially refined {{\(s\)}}-invariants},
 fjournal = {Pacific Journal of Mathematics},
 journal = {Pac. J. Math.},
 issn = {1945-5844},
 volume = {344},
 number = {1},
 pages = {145--150},
 year = {2026},
 doi = {10.2140/pjm.2026.344.145},
 zbl = {8242593}
}

@article{ren:lee_torus_links,
 author = {Ren, Qiuyu},
 title = {Lee filtration structure of torus links},
 fjournal = {Geometry \& Topology},
 journal = {Geom. Topol.},
 issn = {1465-3060},
 volume = {28},
 number = {8},
 pages = {3935--3960},
 year = {2024},
 doi = {10.2140/gt.2024.28.3935},
 zbMATH = {7960667},
 Zbl = {1560.57016}
}

@article {piccirillo:shake-genus,
    AUTHOR = {Piccirillo, Lisa},
     TITLE = {Shake genus and slice genus},
   JOURNAL = {Geom. Topol.},
  FJOURNAL = {Geometry \& Topology},
    VOLUME = {23},
      YEAR = {2019},
    NUMBER = {5},
     PAGES = {2665--2684},
      ISSN = {1465-3060,1364-0380},
   MRCLASS = {57K10 (57R65)},
  MRNUMBER = {4019900},
MRREVIEWER = {Laurence\ R.\ Taylor},
       DOI = {10.2140/gt.2019.23.2665},
       URL = {https://doi.org/10.2140/gt.2019.23.2665},
 Zbl = {1464.57005},
}

@article{SATO2018113,
 author = {Sato, Kouki},
 title = {A full-twist inequality for the {{\(\nu^{+}\)}}-invariant},
 fjournal = {Topology and its Applications},
 journal = {Topology Appl.},
 issn = {0166-8641},
 volume = {245},
 pages = {113--130},
 year = {2018},
 doi = {10.1016/j.topol.2018.06.010},
 zbMATH = {6910247},
 Zbl = {1406.57008}
}

@article{hom_wu,
 author = {Hom, Jennifer and Wu, Zhongtao},
 title = {Four-ball genus bounds and a refinement of the {Ozsv{\'a}th}-{Szab{\'o}} tau invariant},
 fjournal = {The Journal of Symplectic Geometry},
 journal = {J. Symplectic Geom.},
 issn = {1527-5256},
 volume = {14},
 number = {1},
 pages = {305--323},
 year = {2016},
 doi = {10.4310/JSG.2016.v14.n1.a12},
 zbMATH = {6623437},
 Zbl = {1348.57023}
}

@article {Tagami,
 author = {Tagami, Keiji},
 title = {On annulus presentations, dualizable patterns and {RGB}-diagrams},
 fjournal = {Journal of Knot Theory and its Ramifications},
 journal = {J. Knot Theory Ramifications},
 issn = {0218-2165},
 volume = {33},
 number = {9},
 pages = {23 p., Id/No 2497001},
 year = {2024},
 doi = {10.1142/S0218216524970010},
 zbMATH = {7935791},
 Zbl = {1553.57008}
}

@article{qin,
 author = {Qin, Qianhe},
 title = {An {RBG} construction of integral surgery homeomorphisms},
 fjournal = {Algebraic \& Geometric Topology},
 journal = {Algebr. Geom. Topol.},
 issn = {1472-2747},
 volume = {25},
 number = {6},
 pages = {3755--3774},
 year = {2025},
 doi = {10.2140/agt.2025.25.3755},
 zbMATH = {8125803},
 Zbl = {1577.57007},
}

@book{gompf-stipsicz:book,
	  author = {Gompf, R. E. and Stipsicz, A. I. },
 title = {4-manifolds and {Kirby} calculus},
 fseries = {Graduate Studies in Mathematics},
 series = {Grad. Stud. Math.},
 issn = {1065-7339},
 volume = {20},
 isbn = {0-8218-0994-6},
 year = {1999},
 publisher = {Providence, RI: American Mathematical Society},
 zbMATH = {1356915},
 Zbl = {0933.57020}
}

@phdthesis{collins_thesis,
      title={On the concordance orders of knots}, 
      author={Julia Collins},
      year={2011},
	school={University of Edinburgh},
      arxiv={1206.0669}, 
}

@article{zbMATH07408046,
 author = {Hayden, Kyle and Mark, Thomas E. and Piccirillo, Lisa},
 title = {Exotic {Mazur} manifolds and knot trace invariants},
 fjournal = {Advances in Mathematics},
 journal = {Adv. Math.},
 issn = {0001-8708},
 volume = {391},
 pages = {30 p., Id/No 107994},
 year = {2021},
 doi = {10.1016/j.aim.2021.107994},
 zbMATH = {7408046},
 Zbl = {1482.57019}
}

@article{os:rational_surgeries,
 author = {Ozsv{\'a}th, Peter S. and Szab{\'o}, Zolt{\'a}n},
 title = {Knot {Floer} homology and rational surgeries},
 fjournal = {Algebraic \& Geometric Topology},
 journal = {Algebr. Geom. Topol.},
 issn = {1472-2747},
 volume = {11},
 number = {1},
 pages = {1--68},
 year = {2011},
 doi = {10.2140/agt.2011.11.1},
 zbMATH = {5837746},
 Zbl = {1226.57044}
}

@article{os:conc_homos,
 author = {Ozsv{\'a}th, Peter S. and Stipsicz, Andr{\'a}s I. and Szab{\'o}, Zolt{\'a}n},
 title = {Concordance homomorphisms from knot {Floer} homology},
 fjournal = {Advances in Mathematics},
 journal = {Adv. Math.},
 issn = {0001-8708},
 volume = {315},
 pages = {366--426},
 year = {2017},
 doi = {10.1016/j.aim.2017.05.017},
 zbMATH = {6741267},
 Zbl = {1383.57020}
}

@article{os:d_inv,
 author = {Ozsv{\'a}th, Peter and Szab{\'o}, Zolt{\'a}n},
 title = {Absolutely graded {Floer} homologies and intersection forms for four-manifolds with boundary},
 fjournal = {Advances in Mathematics},
 journal = {Adv. Math.},
 issn = {0001-8708},
 volume = {173},
 number = {2},
 pages = {179--261},
 year = {2003},
 doi = {10.1016/S0001-8708(02)00030-0},
 zbMATH = {1891234},
 Zbl = {1025.57016}
}

@article{zbMATH05614872,
 author = {Lobb, Andrew},
 title = {A slice genus lower bound from {{\(sl(n)\)}} {Khovanov}-{Rozansky} homology},
 fjournal = {Advances in Mathematics},
 journal = {Adv. Math.},
 issn = {0001-8708},
 volume = {222},
 number = {4},
 pages = {1220--1276},
 year = {2009},
 doi = {10.1016/j.aim.2009.06.001},
 zbMATH = {5614872},
 Zbl = {1200.57011}
}

@article{zbMATH05550658,
 author = {Wu, Hao},
 title = {On the quantum filtration of the {Khovanov}-{Rozansky} cohomology},
 fjournal = {Advances in Mathematics},
 journal = {Adv. Math.},
 issn = {0001-8708},
 volume = {221},
 number = {1},
 pages = {54--139},
 year = {2009},
 doi = {10.1016/j.aim.2008.12.003},
 zbMATH = {5550658},
 Zbl = {1167.57007}
}

@article{zbMATH06366498,
 author = {Hom, Jennifer},
 title = {Bordered {Heegaard} {Floer} homology and the tau-invariant of cable knots},
 fjournal = {Journal of Topology},
 journal = {J. Topol.},
 issn = {1753-8416},
 volume = {7},
 number = {2},
 pages = {287--326},
 year = {2014},
 doi = {10.1112/jtopol/jtt030},
 zbMATH = {6366498},
 Zbl = {1368.57002}
}

@article{zbMATH06700148,
 author = {Levine, Adam Simon},
 title = {Nonsurjective satellite operators and piecewise-linear concordance},
 fjournal = {Forum of Mathematics, Sigma},
 journal = {Forum Math. Sigma},
 issn = {2050-5094},
 volume = {4},
 pages = {47 p., Id/No e34},
 year = {2016},
 doi = {10.1017/fms.2016.31},
 zbMATH = {6700148},
 Zbl = {1369.57015}
}

@article{rasmussen:hi,
 author = {Rasmussen, Jacob},
 title = {Lens space surgeries and a conjecture of {Goda} and {Teragaito}},
 fjournal = {Geometry \& Topology},
 journal = {Geom. Topol.},
 issn = {1465-3060},
 volume = {8},
 pages = {1013--1031},
 year = {2004},
 doi = {10.2140/gt.2004.8.1013},
 url = {https://eudml.org/doc/124193},
 zbMATH = {2105210},
 Zbl = {1055.57010}
}

@article{ni-wu,
 author = {Ni, Yi and Wu, Zhongtao},
 title = {Cosmetic surgeries on knots in {{\(S^3\)}}},
 fjournal = {Journal f{\"u}r die Reine und Angewandte Mathematik},
 journal = {J. Reine Angew. Math.},
 issn = {0075-4102},
 volume = {706},
 pages = {1--17},
 year = {2015},
 doi = {10.1515/crelle-2013-0067},
 url = {resolver.caltech.edu/CaltechAUTHORS:20150421-115848260},
 zbMATH = {6490569},
 Zbl = {1328.57010}
}

@incollection{zbMATH07442350,
 author = {McCoy, Duncan},
 title = {Null-homologous twisting and the algebraic genus},
 booktitle = {2019--20 MATRIX annals},
 isbn = {978-3-030-62496-5; 978-3-030-62499-6; 978-3-030-62497-2},
 pages = {147--165},
 year = {2021},
 publisher = {Cham: Springer},
 doi = {10.1007/978-3-030-62497-2_7},
 zbMATH = {7442350},
 Zbl = {1498.57006}
}

@article{zbMATH06361412,
 author = {Hom, Jennifer},
 title = {The knot {Floer} complex and the smooth concordance group},
 fjournal = {Commentarii Mathematici Helvetici},
 journal = {Comment. Math. Helv.},
 issn = {0010-2571},
 volume = {89},
 number = {3},
 pages = {537--570},
 year = {2014},
 doi = {10.4171/CMH/326},
 zbMATH = {6361412},
 Zbl = {1312.57008}
}

@book{K3_list,
 editor = {Baykur, R. {\.I}. and Kirby, R.~ C. and Ruberman, D.},
 title = {{K3}: a new problem list in low-dimensional topology},
 fseries = {Mathematical Surveys and Monographs},
 series = {Math. Surv. Monogr.},
 issn = {0076-5376},
 volume = {295},
 isbn = {978-1-4704-8433-0; 978-1-4704-8528-3},
 year = {2026},
 publisher = {Providence, RI: American Mathematical Society (AMS)},
 doi = {10.1090/surv/295},
 Zbl = {8160053}
}

@article{zbMATH08091067,
 author = {Marian, Mihai},
 title = {A remark on the {Lewark}-{Zibrowius} invariant},
 fjournal = {Pacific Journal of Mathematics},
 journal = {Pac. J. Math.},
 issn = {1945-5844},
 volume = {339},
 number = {1},
 pages = {191--200},
 year = {2025},
 doi = {10.2140/pjm.2025.339.191},
 zbMATH = {8091067},
 Zbl = {1572.57029}
}

@article{zbMATH07880416,
 author = {Kjuchukova, Alexandra and Miller, Allison N. and Ray, Arunima and Sakall{\i}, S{\"u}meyra},
 title = {Slicing knots in definite 4-manifolds},
 fjournal = {Transactions of the American Mathematical Society},
 journal = {Trans. Am. Math. Soc.},
 issn = {0002-9947},
 volume = {377},
 number = {8},
 pages = {5905--5946},
 year = {2024},
 doi = {10.1090/tran/9151},
 zbMATH = {7880416},
 Zbl = {1571.57013}
}

@misc{golla_pinzon-caicedo,
      title={The {C}onway knot has infinite order in smooth concordance}, 
      author={Marco Golla and Juanita Pinzón-Caicedo},
      year={2026},
      arxiv={2609.22378},
      howpublished = {Preprint},
}

@article{zbMATH05274876,
 author = {Shumakovitch, Alexander N.},
 title = {Rasmussen invariant, slice-{Bennequin} inequality, and sliceness of knots},
 fjournal = {Journal of Knot Theory and its Ramifications},
 journal = {J. Knot Theory Ramifications},
 issn = {0218-2165},
 volume = {16},
 number = {10},
 pages = {1403--1412},
 year = {2007},
 doi = {10.1142/S0218216507005889},
 zbMATH = {5274876},
 Zbl = {1148.57011}
}

@misc{arXiv:2602.23442,
 author = {Artem Kotelskiy and Liam Watson and Claudius Zibrowius},
 title = {On mutation invariance in {Khovanov} homology},
 year = {2026},
 howpublished = {Preprint},
 url = {https://arxiv.org/abs/2602.23442},
 arXiv = {2602.23442}
}


\end{document}